\documentclass[11pt]{amsart}
\usepackage{latexsym, amssymb, amscd, amsfonts,pstricks,enumitem,amsmath, young}
\usepackage{latexsym, amssymb, amscd, amsfonts,pstricks,enumitem,amsmath}
\usepackage{amsmath,amsthm,amssymb,mathrsfs}
\usepackage{amsmath, amsfonts, amssymb,amsthm}
\usepackage{amstext}
\usepackage{mathrsfs}
\usepackage{tikz-cd}
\usepackage{fancyhdr}
\theoremstyle{plain}
\newtheorem{theorem}{Theorem}[section]

\newtheorem{proposition}[theorem]{Proposition}

\newtheorem{corollary}[theorem]{Corollary}
\newtheorem{def-thm}[theorem]{Definition-Theorem}
\newtheorem{lemma}[theorem]{Lemma}

\newtheorem{definition}[theorem]{Definition}

\newtheorem*{maintheorem}{Main Theorem}

\theoremstyle{definition}
\newtheorem{remark}[theorem]{Remark}
\newtheorem{example}[theorem]{Example}

\newcommand{\sq}[1]{\ifx#1([\else\ifx#1)]%
  \else\message{invalid use of "sq"}\fi\fi}

\DeclareMathOperator{\Supp}{Supp}

\DeclareMathSymbol{\idot}{\mathbin}{operators}{`\.}
\allowdisplaybreaks
\begin{document}
\title{{The Ru-Vojta result  for subvarieties }}
\author{Min Ru}
\address{
 Department  of Mathematics\newline
\indent University of Houston\newline
\indent Houston,  TX 77204, U.S.A.} 
\email{minru@math.uh.edu}
\author{Julie Tzu-Yueh Wang}
\address{Institute of Mathematics, Academia Sinica \newline
\indent No.\ 1, Sec.\ 4, Roosevelt Road\newline
\indent Taipei 10617, Taiwan}
\email{jwang@math.sinica.edu.tw}
\begin{abstract}  In \cite{RV19}, Min Ru and Paul  Vojta, among other things,  proved the so-called  general theorem (arithmetic part) which can be viewed as an extension of Schmidt's subspace theorem. 
In this note, we extend their result by replacing the divisors by  closed subschemes.  
 \end{abstract}
 \thanks{2010\ {\it Mathematics Subject Classification.}
11J97, 11J87, 14G05.}  
\thanks{The first named author is supported in part by the Simons Foundation
  grant award \#531604. The second named author was supported in part by Taiwan's MoST grant 108-2115-M-001-001-MY2.}

\baselineskip=16truept \maketitle \pagestyle{myheadings}
\markboth{}{A subspace theorem  for subvarieties }

\section{Introduction and Statements}
Let $X$ be a  complete variety. 
Let $\mathscr L$ be a big line sheaf and let $D$ be a nonzero effective
Cartier divisor on  $X$.  We define
\begin{equation*}
  \beta(\mathscr L, D)
    =  \liminf_{N\to\infty}
      \frac {\sum_{m\ge 1}h^0(\mathscr L^N(-mD))} {Nh^0(\mathscr L^N)}\;.
\end{equation*}
In \cite{RV19}, Min Ru and  Paul Vojta proved the following theorem which can be viewed as an extension of the Schmidt's subspace theorem. 

\begin{theorem} [\cite{RV19}, General theorem (Arithmetic Part)] \label{Gaa}
Let $X$ be a projective variety over a number field $k$, and
let $D_1, \dots, D_q$ be nonzero effective Cartier divisors
intersecting properly on $X$.  Let $\mathscr L$ be a big line sheaf on $X$.
Let $S\subset M_k$ be a finite set of places. Then, for every $\epsilon>0$, 
the inequality
\begin{equation*} 
  \sum_{i=1}^q \beta(\mathscr L,D_i)  m_{ D_i,S}( \frak p) 
    \leq (1+\epsilon) h_{\mathscr L}(\frak p)
\end{equation*}
holds for all $k$-rational points   $\frak p$  outside a proper Zariski-closed
subset  of $X$.
\end{theorem}
 Here, $m_{ D,S}( \frak p):=\sum_{v\in S} \lambda_{D,v}( \frak p)$, where  $\lambda_{D,v}$ is a Weil function with respect to the divisor $D$ and the place $v$.
 
The purpose of this note is to extend the above theorem to  closed subschemes.
We first introduce some notations. 
\begin{definition}\label{def_aut_lambda}
Let $\mathscr L$ be a big line sheaf and let $Y$ be a closed subscheme
 on a complete variety $X$.  We define
\begin{equation*}
  \beta(\mathscr L, Y)
    = \liminf_{N\to\infty}
      \frac {\sum_{m\ge 1}h^0(\mathscr L^N\otimes {\mathscr I}_{Y}^{m})} {Nh^0(\mathscr L^N)},
\end{equation*}
where ${\mathscr I}_{Y}$ is the ideal sheaf defining $Y$. 
\end{definition}
\begin{remark}\label{remark}
Let $\pi:\tilde X\to X$ be the blow up along $Y$ and $E$ be the exceptional divisor.
If $X$ is smooth  (or just Cohen Macaulay) and $Y$ is a locally complete intersection, then 
$$
h^0(\mathscr L^N\otimes {\mathscr I}_{Y}^{m})=h^0( \pi^*\mathscr L^N(-mE))
$$
for all $m\ge 1$.
See \cite[Remark 4.3.17]{lazpagI}.  Therefore,  in this case,  we have
\begin{equation*}
  \beta(\mathscr L, Y)=\beta(\mathscr L, E)
    := \liminf_{N\to\infty}
      \frac {\sum_{m\ge 1}h^0( \pi^*\mathscr L^N(-mE))} {Nh^0(\mathscr L^N)}\;.
\end{equation*}
\end{remark}
\setcounter{equation}{0}
The main result of our theorem is stated as follows.
\begin{maintheorem} \label{Ga}
Let $X$ be a projective variety over a number field $k$, and
let $Y_1, \dots, Y_q$ be 
closed  subschems of $X$ over $k$
intersecting properly on $X$.  Let $\mathscr L$ be a big line sheaf on $X$.
Let $S\subset M_k$ be a finite set of places. Then, for every $\epsilon>0$, 
the inequality
\begin{equation} \label{Ga_ineq}
  \sum_{i=1}^q \beta(\mathscr L,Y_i) m_{Y_i,S }( \frak p )
    \leq (1+\epsilon) h_{\mathscr L}(\frak p)
\end{equation}
holds for all $k$-rational points $\frak p$ outside a proper Zariski-closed
subset  of $X$.
\end{maintheorem}
 Here, $m_{ Y,S}( \frak p):=\sum_{v\in S} \lambda_{Y,v}( \frak p)$, where  $\lambda_{Y,v}$ is a Weil function with respect to the closed subscheme $Y$ and the place $v$. 
 Closed subschems $Y_1, \dots, Y_q$ 
are said to be {\it intersecting properly on $X$} if  
for any  subset $I\subset \{1, \dots, q\}$   such that $\cap_{i\in I}   {\rm Supp } Y_i \ne\emptyset $, the sequence 
$\{\phi_{i,1}\hdots\phi_{i, \epsilon_i}, i\in I\}$ is a regular sequence  in the local ring ${\mathcal O}_{\frak p, X }$ for $\frak p\in\cap_{i\in I} {\rm Supp }Y_i$,
where, for each $1\leq i\leq q$,  $\phi_{i,1}\hdots\phi_{i,\epsilon_i}$  are the local functions in  ${\mathcal O}_{\frak p, X }$ defining $Y_i$.  Therefore, $\epsilon_i\ge  {\rm codim} ~Y_i $, and $\sum_{i\in I} \epsilon_i\le \dim X $. 
 
\begin{remark}\label{remark2}
In the same context, we say that $Y_1, \dots, Y_q$  are {\it in general position}   if ${\rm codim}\left(\cap_{i\in I}Y_i\right)\ge \sum_{i\in I}{\rm codim}~Y_i$ for any $I\subset \{1,\hdots,q\}$.  Here we use the convention that $\dim\emptyset=-\infty$.  
If $Y_1, \dots, Y_q$ are of locally complete intersection closed subschemes intersecting properly on $X$, then they are    in general position.  The converse holds if $X$ is Cohen-Macaulay by \cite[Theorem 17.4]{Matsumura}.  
\end{remark}
We note that our   main theorem  extends Theorem \ref{Gaa} and its 
 counter-part  in Nevanlinna theory can also be obtained.  For details  of such correspondence, see \cite{RV19} or \cite{rw}.

We now introduce some background on the study of Schmidt's subspace theorem for closed subschemes.  In \cite{rw}, the authors observed a connection between the work of Mckinnon and Roth in \cite{McRo} and the arithmetic general theorem in \cite{RV19}, as well as provided a simpler proof of the result of Mckinnon-Roth (see \cite{McRo}).   
The following is the statement.
\begin{theorem}[\cite{rw}]\label{re}  Let $X$ be a projective variety defined over $k$ and 
$Y_1,\cdots,Y_q$  be closed subschemes  of $X$ defined over $k$ such that at most $\ell$ of the closed subschemes meet at any point of $X$.    Let $A$ be a big Cartier divisor on $X$ and let
\begin{align}\label{betaAY}
\beta_{A,Y_i}=\lim_{N\to\infty}\frac{\sum_{m=1}^{\infty}h^0(\tilde X_i, N\pi_i^*A-mE_i)}{Nh^0(X,NA)}, \quad1\le i\le q,
\end{align}
where $\pi_i:\tilde X_i\to X$ is the blowing-up of $X$ along $Y_i$, with the associated exceptional divisor $E_i$.
Then, for any $\epsilon>0$,
$$
 \sum_{i=1}^q m_{Y_i,S }( \frak p )\le \ell (\max_{1\le i\le q} \{\beta_{A,Y_i}^{-1}\}+\epsilon) h_{A}(\frak p)
$$
holds for  all $\frak p$ outside a proper Zariski-closed subset $Z$ of  $X(k)$.
\end{theorem}
When the closed subschemes  $Y_i=y_i$ are distinct points in $X$,  one may take $\ell=1$ in Theorem \ref{re}, ane the connection between Theorem \ref{re} and \cite{McRo} is due to the following inequality shown by Mckinnon and Roth: 
$$\beta_{A,y_i}\ge \frac{n}{n+1}\epsilon_{y_i}(A),$$
where $\epsilon_{y_i}(A)$ is the Seshadri constant defined below. 

\begin{definition}
Let $Y$ be a closed subscheme of a projective variety $X$ and let $\pi:\tilde X\to X$ be the blowing-up of $X$ along $Y$ with the exceptional divisor $E$.  Let $A$ be a nef Cartier divisor on $X$.  The 
{\it Seshadri constant} $\epsilon_{Y}(A)$ of $Y$ with respect to $A$ is the real number
$$
 \epsilon_{Y}(A)=\sup\{\gamma\in\mathbb Q_{\ge 0}\, |\, \pi^*A-\gamma E\text{ is }\mathbb Q{\rm -nef}\}. 
$$
We note that when $Y$ is a Cartier divisor on $X$, then $\epsilon_{Y}(A)$ can be defined similarly without blowing up.
\end{definition}

Recently, G. Heier and A. Levin also  generalized the result in \cite{McRo} from points to close subschmes.  
\begin{theorem}[\cite{HL}]\label{HLtheorem}  
Let $X$ be a projective variety over a number field $k$, and
let $Y_1, \dots, Y_q$ be closed  subschems of $X$ defined over $k$, and in general position.
 Let $A$ be an ample Cartier divisor on $X$.
Let $S\subset M_k$ be a finite set of places. Then, for every $\epsilon>0$, 
the inequality
\begin{equation} \label{Ga_ineq}
  \sum_{i=1}^q \epsilon_{Y_i}(A) m_{Y_i,S }( \frak p )
    \leq (n+1+\epsilon) h_{A}(  \frak p )
\end{equation}
holds for all  $\frak p$  outside a proper Zariski-closed
subset $Z$ of $X(k)$.
\end{theorem}
 They also proved (\cite[Theorem 4.2]{HL})
\begin{align}\label{compare}
\beta_{A,Y}\ge \frac{r}{n+1} \epsilon_{Y}(A),
\end{align}
where $\beta_{A,Y}$ is as in \eqref{betaAY}.
Therefore, when $X$ is smooth and the $Y_i$ are of locally complete intersection,  our main theorem implies 
Theorem \ref{HLtheorem} by Remark \ref{remark} and \ref{remark2}.  Here, we let  $\mathscr L=\mathcal O(A)$.

Note that the statement of Theorem \ref{re} is weaker than the Main Theorem although  Theorem \ref{re} dealt with more general situation that 
``at most $\ell$ of the closed subschemes meet at any point of $X$".   Theorem \ref{HLtheorem}  is also special since it only deals with  
the Seshadri constant  $\epsilon_{Y_i}(A)$ rather than $\beta( \mathcal O(A),Y_i)$.  

The following is an example coming from \cite[Corollary 2]{cz_am}.  We use it  to compare these two constants.     
\begin{example}\label{example}
Let $L_1, L_2, L_3, L_4$ be four lines in general position  $\mathbb P^2$ and choose three points $P_1,P_2,P_3$, with $P_i\in L_i$ for $i=1,2,3$, $P_i\notin L_j$ if $j\ne i$.  Let $X\to\mathbb P^2$ be the blow-up of the three points $P_1, P_2,P_3$ and let $D_i$, $1\le i\le 4$ be the strict transform of $L_i$. Let $A=\ell D_1+\ell D_2+\ell D_3+D_4$, where $\ell$ is a  positive integer.  It is easy to compute that $A^2=6\ell^2+6\ell+1$, $(A.D_i)=2\ell+1$  for $1\le i\le 3$.  (See also  \cite[Proposition 2]{cz_am}.)

 It is explained in  the proof of \cite[Lemma 5.6]{rtw}  that
 $$\beta( \mathcal O(A),D_i)=\frac{ \frac23\xi_iA^2-\frac13(A.D_i)\xi_i^2 }{A^2}.$$
Here, $\xi_i$ is the constant such that $(A-\xi_iD_i)^2=0$ and $\xi_i=\frac{A^2}{2(A.D_i)}$  for $1\le i\le 3$. (See also  \cite[Proposition 2]{cz_am})
 These datas give
$$\beta( \mathcal O(A),D_i)=\frac{6\ell^2+6\ell+1}{8\ell+4}\ge \frac{3\ell}4.$$ 
On the other hand, it is clear that $A$ is ample and the Seshadri constant  $\epsilon_{D_i}(A)=\ell$ for $1\le i\le 3$.  
Then the right hand side of   \eqref{compare} is $\frac{\ell}{3}$ and left hand side is at least $\frac{3\ell}4$.
\end{example}  
The background materials will be given in the next section.  In Section \ref{filtration}, we will recall some definitions and results from  \cite{Aut2} and prove our main lemma that enables us to generalize the key proposition  in   \cite{Aut2} for  closed subschemes.  The proof of the Main Theorem will be given in the last section.
 \section{Preliminary }
\noindent{\bf Number field and valuations}. 
For a number field $k$, recall that $M_k$ denotes the set of places of $k$, and
that $k_\upsilon$ denotes the completion of $k$ at a place $\upsilon\in M_k$.
Norms $\|\cdot\|_\upsilon$ on $k$ are normalized so that
$$\|x\|_\upsilon = |\sigma(x)|^{[k_\upsilon:\mathbb R]} \qquad\text{or}\qquad
  \|p\|_\upsilon = p^{-[k_\upsilon:\mathbb Q_p]}$$
if $\upsilon\in M_k$ is an archimedean place corresponding to an embedding
$\sigma\colon k\hookrightarrow\mathbb C$ or a non-archimedean place lying
over a rational prime $p$, respectively.

An $M_k$-constant is a collection $(c_v)_{v\in M_k}$ of real constants such
that $c_v=0$ for all but finitely many $v$.
Heights are logarithmic and relative to the number field used as a base field,
which is always denoted by $k$.  For example, if $P$ is a point
on $\mathbb P^n_k$ with homogeneous coordinates $[x_0:\dots:x_n]$ in $k$, then
$$h(P) = h_{\mathscr O(1)}(P)
  = \sum_{\upsilon\in M_k} \log\max\{\|x_0\|_\upsilon,\dots,\|x_n\|_\upsilon\}
  \;.$$

\noindent{\bf Weil function for subschemes}. 
We first recall some basic properties of local Weil functions associated to closed subschemes from \cite[Section 3]{rw} or  \cite[Section 2]{Sil}.
We assume that the readers are familiar with the notion of  Weil functions associated to divisors. (See \cite[Chapter 10]{Lang},  or \cite[Section 1]{Sil}.)

 Let $Y$ be a closed subscheme on a projective variety $V$ defined over $k$. 
It uniquely corresponds to  the ideal sheave ${\mathscr I}_{Y}$. 
J. Silverman showed in \cite{Sil} that one can assign to $Y$,  
for each place $v\in M_k$,  a Weil function (or local height)
$$
\lambda_{Y,v}: V\setminus {\rm supp}(Y)\to \mathbb R
$$
satisfying some functorial properties (up to an $M_k$-constant) described 
in \cite[Theorem 2.1]{Sil}. 
We state some properties which we needed.

\begin{lemma}[\cite{{Sil}}, Theorem 2.1]\label{s}
The Weil function $\lambda$ satisfy the following condtions(up to  $M_k$-constants):

\noindent $($a$)$ If $D$ is an effective divisior, the $\lambda_D$ is is the usual Weil
function associated to the divisors.

\noindent $($b$)$ If $X, Y$ are two closed subschemes, then 
$\lambda_{X\cap Y} = \min\{\lambda_X, \lambda_Y\}$.

\noindent $($c$)$
If $X, Y$ are two closed subschemes, then 
$\lambda_{X+Y}=\lambda_X +\lambda_Y$. In particlar, we have 
$\lambda_{{\mathscr I}^m_{Y}}= m \lambda_{{\mathscr I}_{Y}}$ for any positive 
integer $m$. Here $\lambda_{{\mathscr I}_{Y}}: =\lambda_Y$.

\noindent $($d$)$
If $X, Y$ are two closed subschemes with ${\mathscr I}_{Y}\subset 
{\mathscr I}_{X}$, then  $\lambda_X\leq \lambda_Y$.
\end{lemma}

The Weil function can be expressed locally by the generators as follows.
\begin{lemma}[\cite{{Sil}}, Propsition 2.4] \label{representation}
Let $Y$ be a closed subscheme of $V$. Let $U\subset V$ be an affine open subset such that 
${\mathscr I}_{Y}|_U$ is generated by the sections $\phi_1, \dots, \phi_r\in \Gamma(U, {\mathscr I}_{Y})$. Then 
$$\lambda_{Y,v} = \min_{1\leq i\leq r} \left\{-\log\|\phi_i\|_{\upsilon}\right\} + \alpha_{\upsilon},$$
where $\alpha$ is an $M_k$-bounded function. 
\end{lemma}
 \section{The construction of the filtration and the main lemma}\label{filtration}
 We first recall some definitions and results from  \cite{Aut2} and \cite{RV19}.  We will then show our main lemma that enables us to generalize the key proposition  in   \cite{Aut2} for  closed subschemes.

\begin{definition} Let ${\mathbb N}$ be the set of natural numbers.  A subset $N\subset \mathbb N^r$ is said to be
\textbf{saturated} if ${\bf a}+{\bf b}\in N$ for any 
${\bf a}\in  \mathbb N^r$ and ${\bf b} \in N$.
\end{definition}

\begin{lemma}[Lemma 3.2, \cite{Aut2}]\label{lemm_aut2_3_2}
Let $A$ be a local ring and $(\phi_1, \dots, \phi_r)$ be a regular sequence of 
$A$. Let $M$ and $N$ be two saturated subsets of $\mathbb N^r$.  Then
$${\mathcal I}(M)\cap {\mathcal I}(N)={\mathcal I}(M\cap N),$$
where, for $N\subset \mathbb N^r$, ${\mathcal I}(N)$ is the ideal of $A$ generated by 
$\{\phi_1^{b_1}\cdots \phi_r^{b_r}~|~{\bf b}\in N\}$. 
\end{lemma}

Let $\square_\ell = (\mathbb R^{+})^\ell\setminus\{\boldsymbol 0\}$ for  $\ell\in \mathbb N$.
For each ${\bf t}\in \square_\ell $ and $x\in \mathbb R^+$, let 
$$N_\ell ({\bf t}, x)=\{{\bf b}\in \mathbb N^\ell  \mid t_1b_1+\cdots +t_\ell b_\ell \ge x\}.$$
Let $m$, $r$ and $\epsilon_1,\hdots,\epsilon_r$ be positive integers such that $\sum_{j=1}^r\epsilon_j=m$  (In application, we will use $m\le \dim X$). 
Let $A$ be a local ring and $(\phi_{11}, \dots, \phi_{1\epsilon_1}, \dots, \phi_{r1}, \dots, \phi_{r\epsilon_r})$ be a regular sequence of $A$. 
Then ${\mathcal I}(N_m ({\bf t}, x))$ is the ideal of $A$ generated by 
$$
\{ \phi_{11}^{b_{11}}\cdots \phi_{1\epsilon_1}^{b_{1\epsilon_1}}\cdots\phi_{r1}^{b_{r1}}\cdots \phi_{r\epsilon_r}^{b_{r\epsilon_r}}~|~(b_{11}, \dots, b_{1\epsilon_1}, \dots, b_{r1}, \dots, b_{r\epsilon_r})\in N_m ({\bf t}, x)\}.
$$
Let $I_j$ be the ideal of $A$ generated by $\phi_{j1}, \dots, \phi_{j\epsilon_{j}}$, where $1\le j\le r$.
Denote by ${\mathcal I}_I(N_r ({\bf t}, x))$ the ideal of $A$ generated by 
$$
\{ I_{1}^{b_{1}}\cdots I_{r}^{b_{r}} ~|~{\bf b}\in N_r ({\bf t}, x)\}.
$$ 
The following is our main lemma.
\begin{lemma}\label{switch}
Let ${\bf t}=(t_1,\hdots,t_r)\in \square_r $ and $x\in \mathbb R^+$.  Then
$$
{\mathcal I}_I(N_r ({\bf t}, x))=  {\mathcal I}(N_m ({\bf \tilde t}, x)),
$$
where ${\bf \tilde t}=(t_1,\hdots,t_1,\hdots,t_r,\hdots,t_r)\in \square_m$ with each $t_j$ repeating $\epsilon_j$ times.
\end{lemma}
\begin{proof}
It suffices to check the inclusions on generators on both ideals.
The generators of ${\mathcal I}_I(N_r ({\bf t}, x))$ come  from 
$I_{1}^{b_{1}}\cdots I_{r}^{b_{r}}$ with ${\bf b}\in N_r ({\bf t}, x)$.
Therefore, they are in the following form:
\begin{align}\label{generators}
\phi_{11}^{b_{11}}\cdots \phi_{1\epsilon_1}^{b_{1\epsilon_1}}\cdots\phi_{r1}^{b_{r1}}\cdots \phi_{r\epsilon_r}^{b_{r\epsilon_r}}, 
\end{align}
where the $b_{ji}$ are nonnegative integers such that $b_{j1}+\cdots+b_{j\epsilon_j}=b_j$ for $1\le j\le r$.
Since ${\bf b}\in N_r ({\bf t}, x)$, we see that
\begin{align*}
 t_1b_{11}+\cdots  + t_1b_{1\epsilon_1} +\cdots   +t_rb_{r1}+\cdots +t_r b_{r\epsilon_r}=t_1b_1+\cdots+b_rt_r\ge x.  
\end{align*}
Consequently,
$(b_{11},\hdots,b_{1\epsilon_1},\hdots,b_{r1},\hdots,b_{r\epsilon_r})\in N_m ({\bf \tilde t}, x)$.  This shows that 
${\mathcal I}_I(N_r ({\bf t}, x))$ is a subset of ${\mathcal I}(N_m ({\bf \tilde t}, x))$.

On the other hand, the generator of 
${\mathcal I}(N_m ({\bf \tilde t}, x))$ are of the  form as in \eqref{generators} with
\begin{align}\label{tj}
 t_1b_{11}+\cdots   +t_1b_{1\epsilon_1} +\cdots   +t_rb_{r1}+\cdots +t_r b_{r\epsilon_r}\ge x.  
\end{align}
For each $1\le j\le r$, we let $b_j=b_{j1}+\cdots+b_{j\epsilon_j}$.  Then
$\phi_{j1}^{b_{j1}}\cdots \phi_{j\epsilon_j}^{b_{j\epsilon_j}}\in I_j^{b_j}$  and
$t_1b_1+\cdots+b_rt_r\ge x$ by \eqref{tj}.  This shows that ${\mathcal I}(N_m ({\bf \tilde t}, x))$ is contained in ${\mathcal I}_I(N_r ({\bf t}, x))$.
\end{proof}
\begin{corollary}\label{maincor}
Let ${\bf t}, {\bf u}\in \square_r $ and $x,y\in \mathbb R^+$.  Then
$$
{\mathcal I}_I(N_r ({\bf t}, x))\cap {\mathcal I}_I(N_r ({\bf u}, x))\subset {\mathcal I}_I(N_r (\lambda{\bf t}+(1-\lambda){\bf u}, \lambda x+(1-\lambda)y ))
$$
 for all $ \lambda\in[0,1].$
\end{corollary}
\begin{proof}
By Lemma \ref{switch}, Lemma \ref{lemm_aut2_3_2} and that
$N_m ({\bf \tilde t}, x)\cap N_m ({\bf \tilde u}, y)\subset N_m(\lambda{\bf \tilde t}+(1-\lambda){\bf \tilde u}, \lambda x+(1-\lambda)y) $  for all $ \lambda\in[0,1]$, we have
\begin{align*}
{\mathcal I}_I(N_r ({\bf t}, x))\cap {\mathcal I}_I(N_r ({\bf u}, x))&={\mathcal I}(N_m ({\bf \tilde t}, x))\cap {\mathcal I}(N_m ({\bf \tilde u}, y))\cr
&\subset {\mathcal I}(N_m(\lambda{\bf \tilde t}+(1-\lambda){\bf \tilde u}, \lambda x+(1-\lambda)y) )\\
&={\mathcal I}_I(N_r(\lambda{\bf   t}+(1-\lambda){\bf   u}, \lambda x+(1-\lambda)y) ).
\end{align*}
\end{proof}

We will construct a filtration corresponding to the closed   subschemes  $Y_i$, $1\le i\le q$.  We recall the following definition and result concerning filtration.

\begin{definition} Let $W$ be a vector space of finite dimension.
A \textbf{filtration} of $W$ is a family of subspaces
${\mathcal F}=({\mathcal F}_x)_{x\in \mathbb R^+}$ of subspaces of 
$W$ such that $\mathcal F_x\supseteq\mathcal F_y$ whenever $x\le y$,
and such that ${\mathcal F}_x=\{0\}$ for $x$ big enough.
A basis ${\mathcal B}$ of $W$
is said to be \textbf{adapted to $\mathcal F$} if 
${\mathcal B}\cap {\mathcal F}_x$ is a basis of ${\mathcal F}_x$ for every real number $x\ge 0$.
\end{definition}

\begin{lemma}[Corvaja--Zannier {\cite[Lemma~3.2]{cz_annals}},
  Levin \cite{levin_annals}, Autissier \cite{Aut2}]\label{lemm_filt}
Let ${\mathcal F}$ and ${\mathcal G}$ be two filtrations of $W$.
Then there exists a basis 
of $W$ which is adapted to both ${\mathcal F}$ and ${\mathcal G}$.
\end{lemma}

Let $\mathscr L$ be a big line sheaf, and let $Y_1,\dots,Y_r$ be   closed subschemes of $X$ over $k$ intersecting properly on $X$
and that $\cap_{i=1}^r  \mbox{Supp} Y_j\ne\emptyset$.
For any fixed ${\bf t}\in \square_r$, we construct a filtration of $H^0(X, \mathscr L)$ as follows: 
for  $x\in  \mathbb R^+$, 
one defines the ideal  sheaf  ${\mathscr I}({\bf t}, x)$ of ${\mathscr O}_X$ by
\begin{equation}\label{def_I}
 {\mathscr I}({\bf t}, x)
    = \sum_{{\bf b}\in N_r ({\bf t}, x)} \otimes_{i=1}^r{\mathscr I}_{Y_i}^{b_i}\;,
\end{equation}
and let 
\begin{equation}\label{def_filtr_h_0}
  {\mathcal F}({\bf t})_x
    = H^0(X, \mathscr L \otimes {\mathscr I}({\bf t}, x))\;.
\end{equation}
Then $({\mathcal F}({\bf t})_x)_{x\in \mathbb R^+}$ is a filtration of $H^0(X, \mathscr L)$.

For $s\in H^0(X, \mathscr L) -\{0\}$, let $\mu_{\bf t}(s)=\sup\{y\in \mathbb R^+~|~s\in {\mathcal F}({\bf t})_y\}.$
Also let
\begin{equation}\label{def_F}
  F({\bf t})
    = \frac1{h^0(\mathscr L)}\int_0^{+\infty}(\dim{\mathcal F}({\bf t})_x)\,dx\;.
\end{equation}

Note that, for all $u>0$ and all $\mathbf t\in\square_r$, 
we have $N_r(u\mathbf t,x) = N_r(\mathbf t,u^{-1}x)$, which implies
$\mathcal F(u\mathbf t)_x = \mathcal F(\mathbf t)_{u^{-1}x}$, and therefore
\begin{equation}\label{how_F_scales}
  F(u\mathbf t)
    = \int_0^\infty
      \frac{\dim\mathscr F(\mathbf t)_{u^{-1}x}}{h^0(\mathscr L)}\,dx
    = u\int_0^\infty \frac{\dim\mathscr F(\mathbf t)_y}{h^0(\mathscr L)}\,dy
    = uF(\mathbf t)\;.
\end{equation}

\begin{remark}\label{remk_levin_aut}
Let ${\mathcal B}=\{s_1, \dots, s_l\}$ be a basis of $H^0(X, \mathscr L)$ with $l=h^0(\mathscr L)$. Then we have
$$F({\bf t})\ge \frac1l \int_0^{\infty} \#({\mathcal F}({\bf t})_x\cap {\mathcal B})dx=\frac1l \sum_{k=1}^l \mu_{{\bf t}}(s_k),$$
where equality holds if  ${\mathcal B}$ is adapted to the filtration $({\mathcal F}({\bf t})_x)_{x\in \mathbb R^+}.$
\end{remark}
The key result we will use about this filtration is the following Proposition, which is a generalization of  Th\'eor\`eme 3.6 in \cite{Aut2}. 

\begin{proposition}\label{aut2_thm_3_6}
With the notations and assumptions above, let
$F: \square_{r} \rightarrow \mathbb R^+$ be the map defined in (\ref{def_F}).
Then $F$ is concave.
In particular, for all $\beta_1,\dots,\beta_r\in(0,\infty)$
and all ${\bf t} \in \square_{r}$ satisfying $\sum_{i=1}^r \beta_i t_i=1$,
\begin{equation}\label{aut_ineq}
  F({\bf t})
    \ge \min_i \left(\frac1{\beta_i} \sum_{m\ge 1}
      \frac{h^0(\mathscr L\otimes {\mathscr I}_{Y_i}^{m})}{h^0(\mathscr L)}\right)\;.
\end{equation}
\end{proposition}

 The proof is almost the same as for Th\'eor\`eme 3.6 in \cite{Aut2} once Corollary \ref{maincor} is established.
We include a proof here for the sake of completeness.
\begin{proof} For any ${\bf t}, {\bf u}\in \square_r$ and $\lambda \in [0, 1]$, we need to prove that
\begin{equation}F(\lambda {\bf t}+(1-\lambda){\bf u})\ge \lambda F({\bf t})+(1-\lambda)F({\bf u}).\end{equation}
By Lemma \ref{lemm_filt}, there exists a basis ${\mathcal B}=\{s_1, \dots, s_l\}$ of $H^0(X,\mathscr L)$ with $l=h^0(\mathscr  L)$, 
which is adapted both to $({\mathcal F}({\bf t})_x)_{x\in \mathbb R^+}$
and to $({\mathcal F}({\bf u})_y)_{y\in \mathbb R^+}$.
Let $\frak p\in \cap_{i=1}^r \mbox{Supp} Y_j$ be a point and $\mathcal O_{\frak p,X}$ be the local ring of $X$ at $\frak p$ and 
$\phi_{i,1}\hdots\phi_{i,\epsilon_i}$ be the local defining function of $Y_i$ at $\frak p$.  By our assumption,
$(\phi_{11}, \dots, \phi_{1\epsilon_1}, \dots, \phi_{r1}, \dots, \phi_{r\epsilon_r})$ is a regular sequence of $\mathcal O_{\frak p,X}$.
Hence,  for $x, y\in \mathbb R^+$, by Corollary \ref{maincor},  
$${\mathcal F}({\bf t})_x\cap {\mathcal F}({\bf u})_y
\subset 
{\mathcal F}(\lambda {\bf t}+(1-\lambda){\bf u})_{\lambda x+ (1-\lambda)y}.$$
For $s\in H^0(X, \mathscr L)-\{0\}$, we have, from the definition of $\mu_{{\bf t}}(s)$ and $\mu_{{\bf u}}(s)$, 
$s\in {\mathcal F}(\lambda {\bf t}+(1-\lambda){\bf u})_{\lambda x+ (1-\lambda)y}$ 
for $x<\mu_{{\bf t}}(s)$ and $y<\mu_{{\bf u}}(s)$, and thus
$$\mu_{\lambda {\bf t}+(1-\lambda){\bf u}}(s)\ge \lambda  \mu_{{\bf t}}(s) +(1-\lambda)\mu_{{\bf u}}(s).$$
Taking $s=s_j$ and summing it over $j=1, \dots, l$, we get,
by Remark \ref{remk_levin_aut}, 
$$F(\lambda {\bf t}+(1-\lambda){\bf u})\ge \lambda \frac1l\sum_{j=1}^l  \mu_{{\bf t}}(s_j) +(1-\lambda)\frac1l\sum_{j=1}^l
\mu_{{\bf u}}(s_j).$$
On the other hand, since ${\mathcal B}=\{s_1, \dots, s_l\}$ is a basis
adapted to both ${\mathcal F}({\bf t})$ and ${\mathcal F}({\bf u})$,
from Remark \ref{remk_levin_aut},
$F({\bf t})=\frac1l\sum_{j=1}^l \mu_{{\bf t}}(s_j)$
and $F({\bf u})=\frac1l\sum_{j=1}^l \mu_{{\bf u}}(s_j)$.
Thus $$F(\lambda {\bf t}+(1-\lambda){\bf u})\ge \lambda F({\bf t})+(1-\lambda)F({\bf u}),$$
 which proves that  $F$ is a convex function. 

To prove (\ref{aut_ineq}), let ${\bf e}_1=(1, 0, \dots, 0)$, $\cdots$, 
${\bf e}_r=(0, 0, \dots, 1)$ be the standard basis of $\mathbb R^r$, 
and let $\mathbf t$ be as in (\ref{aut_ineq}).  Then, by convexity of $F$
and by (\ref{how_F_scales}), we get
$$F({\bf t}) \ge \min_i F(\beta_i^{-1}\mathbf e_i)
  = \min_i \beta_i^{-1}F(\mathbf e_i)$$
and, obviously,
$F({\bf e}_i) = \frac1{h^0(\mathscr L)}\sum_{m\ge 1} h^0(\mathscr L\otimes {\mathscr I}_{Y_i}^{m})$
for $i=1, \dots, r$.
\end{proof}
\section{Proof of the Main Theorem }
We first recall the following basic theorem from \cite{RV19}.
\begin{theorem}[\cite{RV19}, Theorem 2.10]\label{schmidt_base}
Let $k$ be a number field, let $S$ be a finite set of places of $k$ containing
all archimedean places, let $X$ be a complete variety over $k$, let $D$ be
a Cartier divisor on $X$,  let $V$ be a nonzero linear subspace of
$H^0(X,\mathscr O(D))$, let $s_1,\dots,s_q$ be nonzero elements of $V$,
let $\epsilon>0$, and let $c\in\mathbb R$.
For each $i=1,\dots,q$, let $D_j$ be the Cartier divisor $(s_j)$,
and let $\lambda_{D_j}$ be a Weil function for $D_j$.
Then there is a proper Zariski-closed subset $Z$ of $X$, depending only
on $k$, $S$, $X$, $L$, $V$, $s_1,\dots,s_q$, $\epsilon$, $c$, and the
choices of Weil and height functions, such that the inequality
\begin{equation}\label{ineq_schmidt_base}
  \sum_{\upsilon\in S}\max_J\sum_{j\in J}\lambda_{D_j,\upsilon}(x)
    \le (\dim V+\epsilon)h_D(x) + c
\end{equation}
holds for all $x\in(X\setminus Z)(k)$.  Here the set $J$ ranges
over all subsets of $\{1,\dots,q\}$ such that the sections $(s_j)_{j\in J}$
are linearly independent.
\end{theorem}


\begin{proof}[Proof of  the Main Theorem]
Let $Y_1,\dots,Y_q$ be 
 closed  subschemes of $X$ over $k$ intersecting properly on $X$, and let $\mathscr L$ be a big line sheaf on $X$.
Recall that $n=\dim X$.  

Let $\epsilon>0$ be given.  Since the quantities $m_{Y_i,S}( \frak p)/h_{\mathscr L}(\frak p)$
are bounded when their respective
denominators are sufficiently large and 
when   $\frak p$ 
lies outside of a proper Zariski-closed subset, it suffices to prove
(\ref{Ga_ineq}) 
with a slightly smaller $\epsilon>0$
and with $\beta(\mathscr L,Y_i)$ replaced by slightly smaller
$\beta_i\in\mathbb Q$ for all $i$.   

Choose $\epsilon_1>0$, and positive integers $N$ and $b$ such that
\begin{equation}\label{aut_choices}
  \left( 1 + \frac nb \right) \max_{1\le i\le q}
      \frac{\beta_i N(h^0(X, \mathscr L^N)+\epsilon_1)}
        {\sum_{m\ge1} h^0(X, \mathscr L^N\otimes {\mathscr I}_{Y_i}^{m})}
    < 1 + \epsilon\;.
\end{equation}

Let
$$\Sigma
  = \biggl\{\sigma\subseteq \{1,\dots,q\}
    \bigm| \bigcap_{j\in \sigma} \Supp Y_j\ne\emptyset\biggr\}\;.$$
For $\sigma\in \Sigma$, let
$$\bigtriangleup_{\sigma}
  = \left\{\mathbf a = (a_i)\in \prod_{i\in\sigma}\beta_i^{-1}\mathbb N
    \Bigm| \sum_{i\in\sigma} \beta_ia_i = b \right\}\;.$$
For $\mathbf a\in\bigtriangleup_{\sigma}$ as above,
one defines 
the ideal sheaf ${\mathscr I}_{\mathbf a}(x)$ of ${\mathscr O}_X$ by
\begin{equation}\label{def_I_of_x}
 {\mathscr I}_{\mathbf a}(x)
    = \sum_{\bf b}   \otimes_{i\in \sigma}{\mathscr I}_{Y_i}^{b_i}
\end{equation}
where the sum is taken for all ${\bf b}\in {\mathbb N}^{\#\sigma}$
with $\sum_{i\in \sigma} a_ib_i\ge bx$, i.e. ${\mathscr I}_{\mathbf a}(x)={\mathscr I}({\bf a}, bx)$ as in \eqref{def_I}.
Let
$${\mathcal F}(\sigma; {\bf a})_x
  = H^0(X, \mathscr L^N \otimes {{\mathscr I}_{\mathbf a}}(x))\;,$$
which we regard as a subspace of $H^0(X,\mathscr L^N)$   (see (\ref{def_filtr_h_0})), and let
$$F(\sigma; {\bf a})
    = \frac1{h^0(\mathscr L^N)}\int_0^{+\infty}(\dim{\mathcal F}(\sigma;{\bf a})_x)\,dx\;$$
     as in  \eqref{def_F}.
Applying Proposition \ref{aut2_thm_3_6} with the line sheaf being taken as $\mathscr L^N$, we have 
$$F(\sigma;\mathbf a)
  \ge \min_{1\leq i\leq q} \left(\frac{b}{\beta_i h^0(\mathscr L^N)}
    \sum_{m\ge 1} h^0(\mathscr L^N\otimes {\mathscr I}_{Y_i}^{m})\right)\;.$$

  As before, for any nonzero $s\in H^0(X,\mathscr L^N)$, we also define
\begin{equation}\label{def_mu_a}
  \mu_{\mathbf a}(s)
    = \sup\{x\in \mathbb R^{+} : s\in \mathcal F(\sigma; \mathbf a)_x\} \;.
\end{equation}

Let $\mathcal B_{\sigma; {\bf a}}$ be a basis of
$H^0(X, \mathscr L^N)$ adapted to the above filtration
$\{{\mathcal F}(\sigma; {\bf a})_x\}_{x\in\mathbb R^+}$.
By Remark \ref{remk_levin_aut},
$F(\sigma; {\bf a}) = \frac1{h^0(\mathscr L^N)}\sum_{s\in \mathcal B_{\sigma; {\bf a}}} \mu_{\mathbf a}(s)$.  Hence
\begin{equation}\label{sum_mu_geq}
 \sum_{s\in \mathcal B_{\sigma; {\bf a}}} \mu_{\mathbf a}(s)
    \ge \min_{1\leq i\leq q}
      \frac b{\beta_i} \sum_{m\ge 1} h^0(\mathscr L^N\otimes {\mathscr I}_{Y_i}^{m})\;.
\end{equation}
It is important to note that there are only finitely many ordered pairs
$(\sigma,\mathbf a)$ with $\sigma\in\Sigma$
and $\mathbf a\in\bigtriangleup_\sigma$.

Let $\sigma\in\Sigma$, $\mathbf a\in\bigtriangleup_\sigma$, and
$s\in H^0(X,\mathscr L^N)$ with $s\ne0$.
We note that  it suffices to use only
the leading terms in (\ref{def_I_of_x}).   The union of the sets of leading
terms as  $x$ ranges over the interval $[0,\mu_{\mathbf a}(s)]$ is finite,
and each such $\mathbf b$ occurs in the sum (\ref{def_I_of_x}) for
a closed set of $x$.  Therefore the supremum (\ref{def_mu_a}) is
actually a maximum. 
We have
\begin{align}\label{chooseK}
\mathscr L^N\otimes\mathscr I_{\mathbf a}(\mu_{\mathbf a}(s))
  = \sum_{\mathbf b\in K} \mathscr L^N  \otimes( \otimes_{i\in\sigma}{\mathscr I}_{Y_i}^{b_i})
  \;,
  \end{align}
where $K=K_{\sigma,\mathbf a,s}$ is the set of minimal elements of
$\{\mathbf b\in\mathbb N^{\#\sigma}
 \mid\sum_{i\in\sigma} a_ib_i\ge  \mu_{\mathbf a}(s)\}$
relative to the product partial ordering on $\mathbb N^{\#\sigma}$.
This set is finite.  Hence, using Lemma \ref{s}, we get
\begin{align}\label{lambdaK}
\lambda_{s,v}(\frak p)\ge \min_{\mathbf b\in K}\sum_{i\in\sigma}b_i\lambda_{Y_i,v}(\frak p)+C_v
\end{align}
for all $\frak p\in X(k)$.
 Write
$$\bigcup_{\sigma; {\bf a}}  \mathcal B_{\sigma; {\bf a}}    = \mathcal B_1\cup\cdots \cup \mathcal B_{T_1}=\{s_1, \dots, s_{T_2}\},$$
 where $\mathcal B_i= \mathcal B_{\sigma; {\bf a}} $ for some $\sigma\in\Sigma$
and $\mathbf a\in\bigtriangleup_\sigma$.
 
 For each $i=1,\dots, T_1$, let $J_i\subseteq\{1,\dots,T_2\}$ be the subset
such that $\mathcal B_i = \{s_j:j\in J_i\}$.  Choose Weil functions
$\lambda_D$, $\lambda_{\mathcal B_i}$ ($i=1,\dots,T_1$),
and $\lambda_{s_j}$ ($j=1,\dots,T_2$) for the divisors $D$, $(\mathcal B_i)$,
and $(s_j)$, respectively.  

\noindent{\bf Claim.}  For each $v\in S$,  
\begin{equation}\label{weilbase2}
  \begin{split}
  &\frac b{b+n} \left( \min_{1\leq i \leq q} \sum_{m\ge 1}
    \frac{h^0(\mathscr L^N\otimes {\mathscr I}_{Y_i}^{m})}{\beta_i} \right)
    \sum_{i=1}^q  \beta_i \lambda_{Y_i, \upsilon}(\frak p)  \\
  &\qquad\le \max_{1\le i\le T_1} \lambda_{\mathcal B_i,v} (\frak p)+ O_{\upsilon}(1)= \max_{1\le i\le T_1} \sum_{j\in J_i} \lambda_{s_j,v} (\frak p)+ O_{\upsilon}(1).
  \end{split}
\end{equation}

First of all, for $\frak p \in X(k)$ and $v\in S$
we want to pick $\sigma_{\frak p ,v}\in\Sigma$ such that
\begin{align}\label{emptyweil1}
\sum_{i=1}^q\beta_i\lambda_{Y_i,v} ( \frak p )\le \sum_{i\in \sigma_{ \frak p } ,v }\beta_i\lambda_{Y_i,v}( \frak p )+O_v(1).
\end{align} 
We can simply take $\sigma_{\frak p ,v}=\{1,\hdots,q\}$ if $\cap_{i=1}^q{\rm Supp}\, Y_i$ is not empty.
Therefore, it suffices to consider when $\cap_{i=1}^q{\rm Supp}\, Y_i=\emptyset$.
Note that for any index subset $I\subset\{1,\hdots,q\}$ such that 
$\cap_{i\in I} {\rm Supp} Y_i=\emptyset$,
$\min_{i\in I}\left\{\lambda_{Y_i,v}\right\}= \lambda_{\cap_{i\in I} Y_{i},v} $
is bounded by an $M_k$-constant.  
Suppose that
 $\lambda_{Y_{i_1},v} ( \frak p )\ge \lambda_{Y_{i_2},v} ( \frak p )\hdots \ge \lambda_{Y_{i_q},v} ( \frak p )$, where $\{i_1,\hdots,i_q\}=\{1,\hdots,q\}$.  Let    $\ell$ be the integer such that
and $\cap_{j=1}^{\ell}  Y_{i_j}$ is not empty and $\cap_{j=1}^{\ell+1}  Y_{i_j}=\emptyset$.
Then $\lambda_{Y_{i_t} ,v} ( \frak p )$, $\ell+1\le t\le q$, is bounded by an $M_k$-constant. 
Therefore, we may take $ \sigma_{ \frak p, v }=\{1,\hdots,\ell\}$.

For $i\in \sigma_{ \frak p,v}$, we let
\begin{align}\label{ti}
t_i:=\frac{\lambda_{Y_i,v}(\frak p )}{\sum_{j\in\sigma_{\frak p ,v}}\beta_j\lambda_{Y_j,v}(\frak p )}. 
\end{align}
Note that $\sum_{i\in \sigma_{\frak p ,v}}\beta_it_i=1$.  Choose  ${\bf a}_{\frak p ,v}=(a_{\frak p ,v;i})\in \bigtriangleup_{\sigma_{\frak p ,v}}$ such that
\begin{align}\label{abn}
 a_{\frak p ,v;i}\le (b+n)t_i\qquad\text{for all } i\in \sigma_{\frak p ,v}.
\end{align}
Then   \eqref{lambdaK}, \eqref{ti}, \eqref{abn}, the definition of $K$ and \eqref{emptyweil1}  imply 
\begin{align*}
\lambda_{s,v}(\frak p )&\ge \min_{\mathbf b\in K}\sum_{i\in\sigma_{\frak p ,v}}b_i\lambda_{Y_i,v}(\frak p )+O_v(1)\cr
&=\big( \sum_{j\in\sigma_{\frak p ,v}}\beta_j\lambda_{Y_j,v}(\frak p )\big) \min_{\mathbf b\in K}\sum_{i\in\sigma_{\frak p ,v}}b_i t_i+O_v(1)\cr
&\ge\big( \sum_{j\in\sigma_{\frak p ,v}}\beta_j\lambda_{Y_j,v}(\frak p )\big) \min_{\mathbf b\in K}\sum_{i\in\sigma_{\frak p ,v}} \frac{a_{\frak p ,v;i} b_i}{b+n}+O_v(1)\cr
&\ge \frac{\mu_{{\bf a}_{\frak p ,v}}(s)}{b+n}   \sum_{j\in\sigma_{ \frak p,v}}\beta_j\lambda_{Y_j,v}(\frak p ) +O_v(1)\cr
&\ge \frac{\mu_{{\bf a}_{\frak p ,v}}(s)}{b+n}   \sum_{i=1}^q\beta_j\lambda_{Y_j,v}(\frak p ) +O_v(1) .
\end{align*}
Then 
\begin{align*}
\sum_{s\in\mathcal B_{\frak p ,\sigma}; {\bf a}} \lambda_{s,v}(\frak p )&\ge
\frac{1}{b+n}(\sum_{s\in\mathcal B_{\frak p ,\sigma}; {\bf a}}  \mu_{{\bf a}_{\frak p ,v}}(s) ) 
  \sum_{i=1}^q\beta_j\lambda_{Y_j,v}(\frak p ) +O_v(1)\cr
&\ge \frac{b}{b+n} \big(\min_{1\leq i\leq q}
      \frac 1{\beta_i} \sum_{m\ge 1} h^0(\mathscr L^N\otimes {\mathscr I}_{Y_i}^{m}) \big) \sum_{i=1}^q\beta_j\lambda_{Y_j,v}(\frak p ) +O_v(1) 
      \end{align*}
by \eqref{sum_mu_geq}.
This proves the Claim.
 
Now,  by Theorem \ref{schmidt_base} with $\epsilon_1$ (as in \eqref{aut_choices}) in place of $\epsilon$,
there  is a proper Zariski-closed subset $Z$ of $X$ such that the inequality
\begin{equation}\label{schmidt_ineq}
  \sum_{v\in S} \max_J \sum_{j\in J} \lambda_{s_j,v}(\frak p )
    \le \left(h^0(\mathscr L^N) + \epsilon_1 \right) h_{\mathscr L^N}(\frak p ) + O(1)
\end{equation}
holds for all $\frak p\in X(k)$ outside of $Z$;
here the maximum is taken over all subsets $J$ of $\{1,\dots,T_2\}$
for which the sections $s_j$, $j\in J$, are linearly independent.

Combining (\ref{weilbase2}) and (\ref{schmidt_ineq}) gives
$$\sum_{i=1}^q \sum_{\upsilon\in S}  \beta_i \lambda_{Y_i, \upsilon}(\frak p )
  \leq \left( 1 + \frac nb \right) \max_{1\le i\le q}
    \frac{\beta_i (h^0(\mathscr L^N) + \epsilon_1)}
      {\sum_{m\ge1} h^0(\mathscr L^N\otimes {\mathscr I}_{Y_i}^{m})} h_{\mathscr L^N}(\frak p ) + O(1)$$
for all $\frak p\in X(k)$ outside of $Z$.  Here we used the fact that all of
the $J_i$ occur among the $J$ in (\ref{schmidt_ineq}).  Using
(\ref{aut_choices}) and the fact that $h_{\mathscr L^N}(\frak p )=Nh_{\mathscr L}(\frak p )$,
we have
\begin{equation*}
  \sum_{i=1}^q \sum_{\upsilon\in S} \beta_i \lambda_{Y_i, \upsilon}(\frak p )
    \leq \left( 1+\epsilon \right) h_{\mathscr L}(\frak p ) + O(1)
\end{equation*}
for all $\frak p \in X(k)$ outside of $Z$.
By the choices of $\beta_i$, this implies that 
 $$ \sum_{i=1}^q \beta(\mathscr L,Y_i) m_{Y_i,S}(\frak p)
    \leq (1+\epsilon) h_{\mathscr L}(\frak p ) + O(1)
$$
holds for all $k$-rational points $\frak p$  outside a proper Zariski-closed set.
\end{proof}

\end{document}